\documentclass[11pt]{article}

\usepackage[T1]{fontenc}
\usepackage{lmodern}
\usepackage{amsmath,amssymb,amsthm,mathtools}
\usepackage{enumitem}
\usepackage{geometry}
\usepackage{microtype}
\usepackage[hidelinks]{hyperref}
\usepackage{indentfirst}
\usepackage{amsmath,amsthm,tikz,float,amssymb}

\hypersetup{
	colorlinks,
	linkcolor={red!90!black},
	citecolor={green!60!black},
	urlcolor={blue!60!black}
}

\allowdisplaybreaks

\newtheorem{theorem}{Theorem}[section]
\newtheorem{lemma}[theorem]{Lemma}

\newtheorem{conjecture}[theorem]{Conjecture}

\newtheorem*{remark}{Remark}

\title{Counterexamples to two conjectures on the diameter of clique-free graphs}
\author{Hangdi Chen\thanks{Email: \texttt{chenhangdi188@126.com}}\\
\small Fujian Key Laboratory of Financial Information Processing, Putian University,\\[-2pt]
\small Putian 351100, P.R. China
\and
Yaojun Chen\thanks{Email: \texttt{yaojunc@nju.edu.cn}}\\
\small School of Mathematics, Nanjing University, Nanjing 210093, P.R. China}
\date{}

\begin{document}

\maketitle

\begin{abstract}
Erd\H{o}s et al. (JCT-B, 1989) conjectured that, for integers $r\ge 2$ and $\delta\ge 2$ with $3r-1\mid\delta$, every connected $K_{2r+1}$-free  graph of order $n$ and minimum degree $\delta$ has diameter at most $ \frac{3r-1}{r}\cdot \frac{n}{\delta}+O(1)$. Czabarka et al. (JCT-B, 2021) later proposed the following generalization: for every $k\ge 3$ and $\delta\ge\left\lceil\frac{3k}{2}\right\rceil-1$, every connected $K_{k+1}$-free graph of order $n$ and minimum degree at least $\delta$ has diameter at most $(3-\frac{2}{k})\cdot\frac{n}{\delta}+O(1)$. We disprove the latter conjecture, including its $k$-colorable version, for every $k\ge 7$ and sufficiently large $\delta$. When $k=2r\ge 8$ and $3r-1\mid\delta$, our construction also disproves the conjecture of Erd\H{o}s et al. (JCT-B, 1989).
\end{abstract}

\section{Introduction}

In this paper, all graphs considered are finite and simple. For a graph $G$, let $V(G)$ and  $E(G)$ denote its \emph{vertex set} and \emph{edge set}, respectively. The \emph{order} of a graph $G$ is $|V(G)|$. For $v\in V(G)$, let $N_G(v)$ be the set of vertices adjacent to $v$, and write $d_G(v)=|N_G(v)|$. The \emph{minimum degree} of $G$ is $\delta (G)=\min\{d_G(v)\mid v\in V(G)\}$. For any $u,v\in V(G)$, the \emph{distance} $d_G(u,v)$ is the minimum length of a $u$--$v$ path in $G$. When $G$ is connected, its \emph{diameter} is $\textup{diam}(G)=\textup{max}_{u,v\in V(G)}~d_G(u,v)$. If no confusion can occur, we will omit the subscript $G$.

The study of diameter bounds in terms of order and  minimum degree goes back to the following classical result, which was proved independently in several papers \cite{Amar1983Germa,Erdos1989Tuza,Goldsmith1981Faber,Moon1965}.

\begin{theorem}\label{ThmConnect}
For every connected graph $G$ of order $n$ and  $\delta(G)=\delta\ge 2$, we have $\textup{diam}(G)\le\frac{3n}{\delta+1}+O(1)$ as $n\to\infty$. 
\end{theorem}

 A graph is \emph{$H$-free} if it does not contain a copy of $H$ as a subgraph. Let $K_n$ denote the \emph{complete graph} on $n$ vertices. A graph is $k$-\emph{colorable} if its vertices admit a proper coloring with at most $k$ colors. The extremal constructions for Theorem \ref{ThmConnect} contain complete subgraphs, and the bound remains sharp even among regular graphs \cite{Cacetta1992Smyth}. Motivated by this observation, Erd\H{o}s, Pach, Pollack, and Tuza \cite{Erdos1989Tuza} proposed stronger bounds for clique-free graphs.

\begin{conjecture}\label{Conj1989Erdos}
\textup{(Erd\H{o}s et al. \cite{Erdos1989Tuza})}.  Let $r,\delta\ge 2$ be two integers and let $G$ be a connected graph of order $n$ with $\delta(G)=\delta$.

\vskip 2mm
\textup{(i)} If $G$ is $K_{2r}$-free and $\delta$ is a multiple of $(r-1)(3r+2)$, then
\[\textup{diam}(G)\le \frac{2(r-1)(3r+2)}{2r^2-1}\cdot \frac{n}{\delta}+O(1)~\textup{as}~n\to\infty.\]

\textup{(ii)} If $G$ is $K_{2r+1}$-free and $\delta$ is a multiple of $3r-1$, then
\[\textup{diam}(G)\le \frac{3r-1}{r}\cdot \frac{n}{\delta}+O(1)~\textup{as}~n\to\infty.\]
\end{conjecture}

Conjecture \ref{Conj1989Erdos} is also recorded in the collection of  Erd\H{o}s problems compiled by Chung and Graham \cite{Chung1998Graham}. The extension of part (ii) to $r=1$, namely the triangle-free case, was proved by  Erd\H{o}s, Pach, Pollack, and Tuza \cite{Erdos1989Tuza}.

 Part (i) is now known to be false. Czabarka, Singgih, and Sz\'{e}kely \cite{Czabarka2021Szekely} constructed counterexamples to part (i) for every $r\ge 2$ and $\delta>2(r-1)(3r+2)(2r-3)$. Cambie and Jooken \cite{Cambie2025Jooken} later extended this negative result to the case $r=2$ and $\delta=2(r-1)(3r+2)(2r-3)$. Motivated by the first family of counterexamples, Czabarka, Singgih, and Sz\'{e}kely \cite{Czabarka2021Szekely} proposed the following modification.

\begin{conjecture}\label{Conj2021Czaba}
\textup{(Czabarka et al. \cite{Czabarka2021Szekely})}.  For every $k\ge 3$ and $\delta\ge\lceil\frac{3k}{2}\rceil-1$, if $G$ is a connected $K_{k+1}$-free (weaker version: $k$-colorable)  graph of order $n$ and $\delta(G)\ge \delta$, then $\textup{diam}(G)\le (3-\frac{2}{k})\cdot\frac{n}{\delta}+O(1)$ as $n\to\infty$. 
\end{conjecture}

For $k=2r$, Conjecture \ref{Conj2021Czaba} has the same forbidden clique and the same coefficient as Conjecture \ref{Conj1989Erdos} (ii). If $\delta$ is a multiple of $3r-1$, Conjecture \ref{Conj2021Czaba} contains the corresponding assertion of Conjecture \ref{Conj1989Erdos} (ii). For $k\in \{3,4\}$, the $k$-colorable version of Conjecture \ref{Conj2021Czaba} was established in \cite{Czabarka2009Szekely,Czabarka2023Szekely}.
\vskip 2mm
In this paper, we disprove Conjecture \ref{Conj2021Czaba}, including its weaker version, for every $k\ge 7$. In Section \ref{erdoCount}, we construct counterexamples to Conjecture \ref{Conj2021Czaba} for every $k=2r\ge 8$ and every $\delta\ge 6(6r-5)(2r-1)(3r-1)$. If $3r-1$ divides $\delta$, these constructions also yield counterexamples to Conjecture \ref{Conj1989Erdos} (ii) for every $r\ge 4$. In Section \ref{czabCount}, we construct counterexamples to Conjecture \ref{Conj2021Czaba} for every $k=2r-1\ge 7$ and $\delta\ge 5k^4$.


At the end of this section, we define the common language for the two constructions. A \emph{layered clique graph} $H$ is one with an ordered partition \[V(H)=L_1\sqcup L_2\sqcup\cdots\sqcup L_t,\]
into nonempty sets, called \emph{layers}, in which two distinct vertices are adjacent if and only if they lie in the same layer or in consecutive layers. Thus $L_i\cup L_{i+1}$ induces a clique for every $i\in\{1,\ldots,t-1\}$, and there are no edges between layers whose indices differ by at least two. A \emph{weighted graph} is a graph in which every vertex is assigned a positive integer weight. For a weighted layered clique graph $H$ and every $\ell\in\{1,\ldots,t\}$, let $A_{\ell}$ denote the total weight of the layer $L_{\ell}$. For any vertex $u\in V(H)$, let $W(u)$ denote the sum of the weights of the neighbors of $u$ in $H$. 
 
 The \emph{blow-up} of a weighted graph $H$ is obtained by replacing every vertex $v\in V(H)$ by an independent set of size equal to its weight and every edge $uv\in E(H)$ by all edges between the two corresponding independent sets. We call the independent set replacing $v$ the \emph{blow-up class} of $v$.

\section{Counterexamples for even $k$}\label{erdoCount}

Fix positive integers $p$, $r$, and $\delta$ such that $r\ge 4$ and \[\delta\ge 6(6r-5)(2r-1)(3r-1).\] Set $\tau=2r-1$, $d=3r-1$, $\lambda=\left\lceil \frac{\delta}{d}\right\rceil$, and $x=\frac{\delta}{\tau}$.
 Furthermore, define $a_3=\lfloor x\rfloor$, $a_1=\lceil x\rceil$,  $\eta=\delta-\lambda-\tau$, $z=a_1+a_3-\lambda-\tau$, and $g=\left\lceil\frac{\eta}{z}\right\rceil$. Finally, for each $i\in \{1,\ldots,2r-3\}$, let $b_i=\delta-\lfloor ix\rfloor$ and $c_i=\lfloor(i+1)x\rfloor$.


We now construct a weighted layered clique graph $J_{p,r}$ with $p(6r-5)+5$ layers, denoted by $L_1,L_2,\ldots,L_{p(6r-5)+5}$. The following list specifies both the number of vertices in each layer and their weights. \\

\noindent\textbf{(A)} Each of $L_1$ and $L_{p(6r-5)+5}$ contains $2r-1$ vertices, all of weight $\delta$. Each of $L_2$ and $L_{p(6r-5)+4}$ consists of a vertex of weight $\delta$.

\noindent\textbf{(B)} For every $i\in\{1,\ldots,2r-2\}$ and $j\in\{0,1,\ldots,p-1\}$,  the layer $L_{3i+j(6r-5)}$ consists of a vertex of weight $1$. In addition, the layer $L_{p(6r-5)+3}$ consists of a vertex of weight $1$.

\noindent\textbf{(C)} For every $i\in\{1,\ldots,2r-3\}$ and $j\in\{0,1,\ldots,p-1\}$, the layer $L_{3i+1+j(6r-5)}$ has $2r-1-i$ vertices. Their weights belong to $\left\{\left\lfloor\frac{b_{i}}{2r-1-i}\right\rfloor,\left\lceil\frac{b_{i}}{2r-1-i}\right\rceil\right\}$ and sum to  $b_{i}$. 

\noindent\textbf{(D)} For every $i\in\{1,\ldots,2r-3\}$ and  $j\in\{0,1,\ldots,p-1\}$, the layer $L_{3i+2+j(6r-5)}$ contains $i+1$ vertices. Their weights belong to $\left\{\left\lfloor\frac{c_i}{i+1}\right\rfloor,\left\lceil\frac{c_i}{i+1}\right\rceil\right\}$ and sum to $c_{i}$. 

\noindent\textbf{(E)} For every $i\in\{1,3\}$ and $j\in\{0,1,\ldots,p-1\}$, the layer $L_{6r-6+i+j(6r-5)}$ contains three vertices. Their weights belong to $\left\{\left\lfloor \frac{a_i}{3}\right\rfloor,\left\lceil \frac{a_i}{3}\right\rceil\right\}$ and sum to $a_i$. 

\noindent\textbf{(F)} For every $j\in\{0,1,\ldots,p-1\}$, the layer $L_{6r-4+j(6r-5)}$ contains $g$ vertices. Their weights belong to $\left\{\left\lfloor\frac{\eta}{g}\right\rfloor,\left\lceil\frac{\eta}{g}\right\rceil\right\}$ and sum to $\eta$. \\

Before proving that every displayed weight is positive, we bound the number of vertices required in the layers of type \textbf{(F)}.
\begin{lemma}\label{lemGBound}
$1\le g\le 2r-3$.
\end{lemma}
\begin{proof}
We first establish the inequalities needed to bound $g$.  Since $a_1\ge a_3$,  we have $z\ge 2a_3-\lambda-\tau$. Moreover,  $r\ge 4$ gives $\tau=2r-1\ge 7$, and hence \[\frac{\tau}{2}-\frac{5}{2}-\frac{2}{\tau}\ge 1-\frac{2}{7}=\frac{5}{7}.\]
Because $d=3r-1<2\tau$ and $\tau-2<\tau$, we also have \[\frac{2(\tau-2)(d+\tau-2)}{\tau}<6\tau.\]
The definition $ \lambda=\left\lceil\frac{\delta}{d}\right\rceil$ yields $(\lambda-1)d<\delta\le \lambda d$. The assumed lower bound on $\delta$,  together with $6r-5=3\tau-2$, gives $\lambda\ge  6(3\tau-2)\tau$. Write $\delta=\tau a_3+\phi$, where $0\le\phi\le \tau-1$. Then \[a_3=\frac{\delta-\phi}{\tau}\ge\frac{(\lambda-1)d+2-\tau}{\tau}.\]
Combining these estimates, we obtain
\begin{align*}
(\tau-2)z-\eta&\ge (\tau-2)(2a_3-\lambda-\tau)-\delta+\lambda+\tau\\
&= (\tau-2)(2a_3-\lambda)-\delta+\lambda-\tau(\tau-3)\\
&\ge \lambda\left(\frac{\tau}{2}-\frac{5}{2}-\frac{2}{\tau}\right)-\frac{2(\tau-2)(d+\tau-2)}{\tau}-\tau(\tau-3)\\
&\ge \frac{30}{7}(3\tau-2)\tau-6\tau-\tau(\tau-3)=\frac{1}{7}\tau(83\tau-81)>0.
\end{align*}
On the other hand, $\lambda\le \frac{\delta}{d}+1$, so
\[\eta\ge \delta-\frac{\delta}{d}-2r\ge 6(3r-2)(6r-5)(2r-1)-2r>0.\]
It follows from $(\tau-2)z>\eta>0$ that $z>0$ and $\frac{\eta}{z}<\tau-2$. Since $g=\left\lceil\frac{\eta}{z}\right\rceil$, we have $1\le g\le \tau-2\le 2r-3$.
\end{proof}

Let $c_0=\lfloor x\rfloor$ and $y=a_3-\lambda-\tau+1$. For every $i\in \{1,\ldots,2r-3\}$, let $t_i=1+c_i-c_{i-1}$.  The following estimates will be used in the degree calculation. 

\begin{lemma}\label{lemWeightBound}
 The following statements hold.

\noindent\textup{(i)} For every $i\in \{1,\ldots,2r-3\}$, $b_i\le (\tau-i)t_i$ and $c_i\le (i+1)t_i$.

\noindent\textup{(ii)}  $a_3<3y$, $a_1\le 3y$ and $\eta\le gz$. 
\end{lemma}
\begin{proof}
For part (i), observe that $c_i-c_{i-1}\ge \lfloor x\rfloor$. So $t_i\ge \lceil x\rceil$. Consequently, $c_i=\lfloor (i+1)x\rfloor\le (i+1)x\le (i+1)t_i$. The integrality of $\delta=\tau x$ gives \[b_i=\delta-\lfloor ix\rfloor=\delta-\lfloor\delta-(\tau-i)x\rfloor=\lceil (\tau-i)x\rceil\le (\tau-i)\lceil x\rceil\le (\tau-i)t_i.\] 

We now prove part (ii). By the proof of Lemma \ref{lemGBound}, we have $d<2\tau$, $\lambda\ge 6(3\tau-2)\tau$ and $a_3\ge \frac{(\lambda-1)d+2-\tau}{\tau}$. Therefore, 
\begin{align*}
3y-a_3&=2a_3-3\lambda-3\tau+3\ge \frac{\lambda}{\tau}-\frac{2(d+\tau-2)}{\tau}-3\tau+3\\
&> 6(3\tau-2)-\frac{2(3\tau-2)}{\tau}-3\tau+3>15\tau-15>0.
\end{align*}
Thus $a_3<3y$. Since $a_3$ and $3y$ are integers, $a_3+1\le 3y$, and hence $a_1\le a_3+1\le 3y$. Finally, Lemma \ref{lemGBound} gives $z>0$ and $\eta>0$. And  $g=\left\lceil\frac{\eta}{z}\right\rceil\ge \frac{\eta}{z}$ implies $\eta\le gz$. 
\end{proof}

We next establish the structural and numerical properties of $J_{p,r}$ needed for the main even-case construction. 

\begin{lemma}\label{lemPositive}
All weights assigned in $J_{p,r}$ are positive integers.
\end{lemma}
\begin{proof}
From the proof of Lemma \ref{lemWeightBound} (i), we have $b_i=\lceil (\tau-i)x\rceil$.
The hypothesis on $\delta$ implies $x\ge 6(6r-5)(3r-1)\ge 1254$. Hence $b_i\ge 2r-1-i$, and therefore $\left\lfloor\frac{b_i}{2r-1-i}\right\rfloor$ and $\left\lceil\frac{b_i}{2r-1-i}\right\rceil$ are positive. Similarly, $c_i=\lfloor(i+1)x\rfloor\ge i+1$.  So the two weights used in part \textbf{(D)} are positive. 

Since $x\ge 1254$, we have $a_1\ge a_3 \ge 3$. So the weights in part \textbf{(E)} are positive. Finally, Lemma \ref{lemGBound} gives $z\ge 1$ and $\eta>0$. Thus $1\le g=\lceil\frac{\eta}{z}\rceil\le \eta$, and the two weights in part \textbf{(F)} are positive as well. 
\end{proof}

We next verify the chromatic property of $J_{p,r}$.
\begin{lemma}\label{lemChromatic}
The graph $J_{p,r}$ is $2r$-colorable.
\end{lemma}
\begin{proof}
Because edges occur only within a layer or between consecutive layers, it is enough to prove $|L_{\ell}\cup L_{\ell+1}|\le 2r$ for every $\ell\in\{1,\ldots,p(6r-5)+4\}$; one may then color the layers greedily from left to right. 

For $\ell\in\{1,2,p(6r-5)+3,p(6r-5)+4\}$, the desired inequality for $L_{\ell}\cup L_{\ell+1}$ follows immediately from parts \textbf{(A)} and \textbf{(B)}. By periodicity, it remains  to consider the case $3\le\ell\le 6r-3$. If $\ell=3i$ with $1\le i\le 2r-3$, then parts \textbf{(B)} and \textbf{(C)} give $|L_{\ell}\cup L_{\ell+1}|=1+(2r-1-i)\le 2r-1$. If $\ell=3i+1$ with $1\le i\le 2r-3$, then parts \textbf{(C)} and \textbf{(D)} give $|L_{\ell}\cup L_{\ell+1}|=(2r-1-i)+(i+1)= 2r$. If $\ell=3i+2$ with $1\le i\le 2r-3$, then parts \textbf{(B)} and \textbf{(D)} give $|L_{\ell}\cup L_{\ell+1}|=(i+1)+1\le 2r-1$. For $\ell\in\{6r-6,6r-3\}$, the relevant union has $4$ $ (\le 2r)$ vertices. Finally, for $\ell\in\{6r-5,6r-4\}$, Lemma \ref{lemGBound} gives $|L_{\ell}\cup L_{\ell+1}|=3+g\le 2r$. 
\end{proof}

 For $j\in\{0,1,\ldots,p-1\}$, call the layers $L_{3+j(6r-5)},L_{4+j(6r-5)},\ldots,L_{6r-3+j(6r-5)}$ the $j$-th \emph{period}. The layer immediately following a period is called its \emph{right-hand junction}. The following lemma establishes the required lower bound on $W(u)$. 

\begin{lemma}\label{lemNeighborSumWeights}
For every vertex $u\in V(J_{p,r})$, we have $W(u)\ge\delta$. Moreover, $W(u)=\delta$ for $u\in L_6$.
\end{lemma}
\begin{proof}
Suppose that $u\in L_{\ell}$ has weight $q$. For every layer other than the two end layers, $W(u)=A_{\ell-1}+A_{\ell}+A_{\ell+1}-q$. Let $T_{\ell}=A_{\ell-1}+A_{\ell}+A_{\ell+1}-\delta$. Thus $W(u)=\delta+T_{\ell}-q$, and it suffices to prove $T_{\ell}\ge q$. Periodicity reduces the internal calculation to one period and its right-hand junction.

 For $\ell=3$, parts \textbf{(A)}--\textbf{(C)} give $T_{\ell}=1+b_1\ge 1\ge q$.  If $\ell=3i$ and $2\le i\le 2r-3$, then $b_i+c_{i-1}=\delta$, and hence $T_{\ell}=c_{i-1}+1+b_i-\delta=1=q$. In particular, $W(u)=\delta$ if $u\in L_6$.  The first layer of any later period has index congruent to $3$ modulo $6r-5$, and the identity $b_1=\delta-a_3$ gives $T_{\ell}=a_3+1+b_1-\delta=1=q$.

If $\ell\in\{3i+1,3i+2\}$ and $1\le i\le 2r-3$, parts \textbf{(B)}--\textbf{(D)} give  $T_{\ell}=1+b_i+c_i-\delta=t_i$. Lemma \ref{lemWeightBound} (i) shows that $t_i$ is at least both $\lceil \frac{b_i}{2r-1-i}\rceil$ and $\lceil \frac{c_i}{i+1}\rceil$. So $T_{\ell}\ge q$. 

It remains to check the four special layers at the end of a period and the final right-hand junction $L_{p(6r-5)+3}$. Since $c_{2r-3}=\lfloor (\tau-1)x\rfloor=\delta-a_1$, we obtain $T_{6r-6}=c_{2r-3}+1+a_1-\delta=1=q$. For $\ell=6r-5$,  Lemma \ref{lemWeightBound} (ii) gives $T_{\ell}=1+a_1+\eta-\delta\ge y\ge \lceil \frac{a_1}{3}\rceil\ge q$. For $\ell=6r-4$, the same lemma gives $T_{\ell}=a_1+\eta+a_3-\delta=z\ge \lceil \frac{\eta}{g}\rceil\ge q$. Finally, for $\ell=6r-3$, it gives $T_{\ell}=\eta+a_3+1-\delta=y\ge \lceil \frac{a_3}{3}\rceil\ge q$. These computations apply to every period after shifting the indices by a multiple of $6r-5$. At the final right-hand junction $\ell=p(6r-5)+3$, we have  $T_{\ell}=a_3+1\ge 1=q$. 

For completeness, we consider the case $\ell\in\{1,2,p(6r-5)+4,p(6r-5)+5\}$. By symmetry, it suffices to consider $\ell\in\{1,2\}$. If $\ell=1$, then $q=\delta$ and $W(u)=A_1+A_2-\delta=(2r-1)\delta\ge\delta$.  If $\ell=2$, then $T_{\ell}=(2r-1)\delta+1\ge\delta= q$. Hence, we complete the proof.
\end{proof}

We conclude the preliminary analysis by determining the total weight of $J_{p,r}$.
\begin{lemma}\label{lemSumWeights}
The total weight of $J_{p,r}$ is $p(\tau\delta-\lambda-1)+4r\delta+1$. 
\end{lemma}
\begin{proof}
 Let $S$ be the total weight of one period. Since $b_i+c_{i-1}=\delta$ for $1\le i\le 2r-3$, $c_0=a_3$, and $c_{2r-3}=\delta-a_1$, the definition of $J_{p,r}$ gives
\begin{align*}
S&=\sum_{i=1}^{2r-3} (1+b_i+c_i)+(1+a_1+\eta+a_3)\\
&=\sum_{i=1}^{2r-3} (1+\delta-c_{i-1}+c_i)+(1+\delta-c_{2r-3}+\eta+c_0)\\
&=(2r-3)(1+\delta)+(1+\delta+\eta)=\tau\delta-\lambda-1.
\end{align*}
The five layers outside the $p$ periods have total weight $4r\delta+1$. Adding their contribution to $pS$ proves the stated formula.
\end{proof}

We now assemble the preceding properties into the even-case counterexamples.

\begin{theorem}\label{ThmExample1.1}
 Let  $p\ge 1$, $r\ge 4$, $\delta\ge 6(6r-5)(2r-1)(3r-1)$ be integers. Set $d=3r-1$, $\lambda=\left\lceil\frac{\delta}{d}\right\rceil$, and $\tau=2r-1$. Then there exists a connected $2r$-colorable (and hence $K_{2r+1}$-free) graph $G_{p,r}$ with minimum degree $\delta$, order $n=p(\tau\delta-\lambda-1)+4r\delta+1$, and diameter $p(6r-5)+4$. Moreover, for every fixed pair $(r,\delta)$, $\textup{diam}(G_{p,r})-\frac{3r-1}{r}\cdot\frac{n}{\delta}\to+\infty$ as $p\to+\infty$. 
\end{theorem}
\begin{proof}
Let $G_{p,r}$ be the blow-up of $J_{p,r}$. The graph $G_{p,r}$ is connected because every layer is nonempty and every two consecutive layers are completely joined. By the definition of blow-up and $J_{p,r}$, the chromatic number of $G_{p,r}$ and $J_{p,r}$ is the same. Lemma \ref{lemChromatic} therefore shows $G_{p,r}$ is $2r$-colorable and hence $K_{2r+1}$-free. 

The degree of a vertex in a blown-up class is exactly  the sum of the weights of the neighbors of the corresponding vertex of  $J_{p,r}$. Lemma \ref{lemNeighborSumWeights} therefore shows that the minimum degree of $G_{p,r}$ is at least $\delta$, and the equality on the blow-up of $L_6$ shows that the minimum degree is exactly $\delta$. Lemma \ref{lemSumWeights} gives the asserted order of  $G_{p,r}$.

 For every $\ell\in\{1,\ldots,p(6r-5)+5\}$, let $L_{\ell}'$ be the union of the blown-up classes corresponding to vertices in $L_{\ell}$. Every edge of $G_{p,r}$ has both ends in the same layer or in consecutive layers, and the consecutive layers are completely joined. Let $i,j$ be integers satisfying $1\le i<j\le p(6r-5)+5$. So every vertex of $L_i'$ is at distance exactly $j-i$ from every vertex of $L_j'$. Two vertices in the same blown-up class have distance at most $2$, while vertices in distinct classes of the same layer are adjacent. Since $p(6r-5)+4\ge 2$, it follows that $\textup{diam}(G_{p,r})=p(6r-5)+4$. 

It remains to compare this diameter with the conjectured bound. Since $\lambda=\left\lceil\frac{\delta}{d}\right\rceil$, we have $\delta\le \lambda d$ and therefore $(\lambda+1)d-\delta\ge d>0$. Using $d=3r-1$ and $\tau=2r-1$, we obtain
\begin{align*}
\textup{diam}(G_{p,r})-\frac{3r-1}{r}\cdot\frac{n}{\delta}&=p\frac{(\lambda+1)d-\delta}{r\delta}+4-\frac{d(4r\delta+1)}{r\delta}\\
&\ge p\frac{d}{r\delta}+4-\frac{d(4r\delta+1)}{r\delta}.
\end{align*}
The right-hand side tends to $+\infty$ with $p$, as required.  
\end{proof}

\begin{remark}
We do not attempt to optimize the lower bound of $\delta$. In fact, a similar construction yields counterexamples to Conjecture \ref{Conj2021Czaba} for every even $k=2r\ge 8$ and every $\delta\ge (6r+72)(3r-1)$ and to Conjecture \ref{Conj1989Erdos} (ii) for every $r\ge 4$ and every $\delta$ in this range that is divisible by $3r-1$. We do not present the details here because it is more involved.

\end{remark}
\section{Counterexamples for odd $k$}\label{czabCount}
We now turn to the remaining parity. Let  $\mu,r,k$, and $\delta$ be positive integers such that $\mu\ge 2$, $r\ge 4$, $k=2r-1$ and $\delta\ge 5k^4$. Define \[a=\frac{1}{k-1}\quad\textup{and}\quad b=1-\frac{2a}{3}=\frac{3k-5}{3(k-1)}.\] 

We construct a weighted layered clique graph $T_{k,\mu}$ with $\mu(3k-5)-1$ layers $L_1,L_2,\ldots,$ $L_{\mu(3k-5)-1}$ as follows. Clearly, all the displayed weights are positive integers.\\

\noindent\textbf{(A)} For every $j\in\{0,1,\ldots,\mu-1\}$, each of $L_{1+j(3k-5)}$ and $L_{3+j(3k-5)}$ has three vertices, each of weight $\left\lceil \frac{\delta a}{3}\right\rceil$. The layer $L_{2+j(3k-5)}$  contains $k-3$ vertices, each of weight $\left\lceil \frac{\delta b}{k-3}\right\rceil$.

\noindent\textbf{(B)} For every $i\in\{1,\ldots,k-3\}$ and $j\in\{0,1,\ldots,\mu-1\}$,  the layer $L_{3i+1+j(3k-5)}$ consists of a vertex of weight $1$.

\noindent\textbf{(C)} For every $j\in\{1,\ldots,\mu-1\}$,  the layer $L_{j(3k-5)}$  consists of a vertex of weight $1$.

\noindent\textbf{(D)} For every $i\in\{1,\ldots,k-3\}$ and $j\in\{0,1,\ldots,\mu-1\}$,  the layer $L_{3i+2+j(3k-5)}$ contains $k-i-1$ vertices, each of weight $\lceil \delta a\rceil$, while the layer $L_{3i+3+j(3k-5})$  contains $i+1$ vertices, each of  weight $\lceil \delta a\rceil$. 
\vskip 2mm
For $j\in\{0,1,\ldots,\mu-1\}$, call the layers $L_{1+j(3k-5)},L_{2+j(3k-5)},\ldots,L_{3k-6+j(3k-5)}$ the $j$-th \emph{basic period}. The layers $L_{j(3k-5)}$ for $1\le j\le\mu-1$ are precisely the junctions between consecutive basic periods. We first establish the structural and numerical properties needed for the odd-case construction.


\begin{lemma}\label{lemTkChromatic}
The graph $T_{k,\mu}$ is $k$-colorable.
\end{lemma}
\begin{proof}
It is enough to show $|L_{\ell}\cup L_{\ell+1}|\le k$ for every $\ell\in\{1,\ldots,\mu(3k-5)-2\}$. By periodicity, consider first $1\le\ell\le3k-5$. If $\ell\in\{1,2\}$, part \textbf{(A)} gives $|L_{\ell}\cup L_{\ell+1}|=k$. If $\ell\in\{3,3k-5\}$, the relevant pair has $4$ $(\le k)$ vertices.

 For $4\le\ell\le 3k-6$, there are three cases. If $\ell=3i+1$ and $1\le i\le k-3$, then $|L_{\ell}\cup L_{\ell+1}|=1+(k-i-1)\le k-1$. If $\ell=3i+2$ and $1\le i\le k-3$, then $|L_{\ell}\cup L_{\ell+1}|=(k-i-1)+(i+1)=k$. If $\ell=3i+3$ and $1\le i\le k-3$, then $|L_{\ell}\cup L_{\ell+1}|=(i+1)+1\le k-1$. These inequalities also cover the junctions between consecutive basic periods. Greedy coloring from left to right now proves the lemma.
\end{proof}
The next lemma establishes the lower bound of $W(v)$. 

\begin{lemma}\label{lemTkWeights}
For every vertex $v\in V(T_{k,\mu})$, we have $W(v)\ge\delta$. 
\end{lemma}
\begin{proof}
Suppose that $v\in L_{\ell}$ has weight $q$. The two end layers can be checked directly. If $\ell=1$, then $$W(v)=2\left\lceil \frac{\delta a}{3}\right\rceil+(k-3)\left\lceil \frac{\delta b}{k-3}\right\rceil\ge\delta\left(\frac{2a}{3}+b\right)=\delta.$$ 
If $\ell=\mu(3k-5)-1$, we have $W(v)=(k-1)\lceil \delta a\rceil\ge\delta(k-1)a=\delta$.

For every non-end layer, $W(v)=A_{\ell-1}+A_{\ell}+A_{\ell+1}-q$. By periodicity, it suffices to check the layer types in the first  basic period and at its right-hand junction. For $\ell=2$, part \textbf{(A)} yields 
\[W(v)=6\left\lceil \frac{\delta a}{3}\right\rceil+(k-4)\left\lceil \frac{\delta b}{k-3}\right\rceil \ge\delta\left(1+\frac{4a}{3}-\frac{b}{k-3}\right)\ge\delta.\]
The last inequality is equivalent to $\frac{4a}{3}\ge\frac{b}{k-3}$, which follows directly from $a=\frac{1}{k-1}$, $b=1-\frac{2a}{3}$, and $k\ge 7$. For $\ell=3$, we have \[W(v)=(k-3)\left\lceil \frac{\delta b}{k-3}\right\rceil+2\left\lceil \frac{\delta a}{3}\right\rceil+1\ge\delta\left(b+\frac{2a}{3}\right)+1=\delta+1.\] 
For the first layer of any later basic period, $\ell=1+j(3k-5)$ with $1\le j\le\mu-1$, we have 
$$W(v)=1+2\left\lceil \frac{\delta a}{3}\right\rceil+(k-3)\left\lceil \frac{\delta b}{k-3}\right\rceil\ge\delta+1.$$ 
For $\ell\in\{4,3k-5\}$, where the latter case corresponds to a junction, parts \textbf{(A)}-\textbf{(D)} give $W(v)=3\left\lceil \frac{\delta a}{3}\right\rceil+(k-2)\lceil \delta a\rceil\ge\delta(k-1)a=\delta$.

Finally, consider the case $5\le\ell\le 3k-6$. If $\ell=3i+1$ with $2\le i\le k-3$, then $W(v)=(k-1)\lceil \delta a\rceil\ge\delta(k-1)a\ge\delta$. If $\ell\in\{3i+2,3i+3\}$ with $1\le i\le k-3$, provided that $L_{\ell}$ is not the final layer handled above, then $W(v)=(k-1)\lceil \delta a\rceil+1>\delta$. This exhausts all layer types.
\end{proof}

Define 
$$w=6\left\lceil \frac{\delta}{3(k-1)}\right\rceil+(k-3)\left\lceil \frac{(3k-5)\delta}{3(k-1)(k-3)}\right\rceil+k(k-3)\left\lceil \frac{\delta}{k-1}\right\rceil+k-2.$$ 
We conclude the preliminary analysis by determining the total weight of $T_{k,\mu}$.

\begin{lemma}\label{lemSum1.2Weights}
The total weight of $T_{k,\mu}$ is $\mu w-1$. 
\end{lemma}
\begin{proof}
Partition the vertices according to parts \textbf{(A)}-\textbf{(D)}, and let $S_1,S_2,S_3,S_4$ be the corresponding total weights. Directly from the definition, $S_1=6\mu\left\lceil\frac{\delta a}{3}\right\rceil+(k-3)\mu\left\lceil\frac{\delta b}{k-3}\right\rceil$, $S_2=(k-3)\mu$, and $S_3=\mu-1$. Moreover, \[S_4=\mu\sum_{i=1}^{k-3}((k-i-1)\lceil\delta a\rceil+(i+1)\lceil\delta a\rceil)=k(k-3)\mu\lceil\delta a\rceil.\]  Substituting $a=\frac{1}{k-1}$ and $b=\frac{3k-5}{3(k-1)}$ into $S_1+S_2+S_3+S_4$ gives $\mu w-1$.
\end{proof}

 Let $H_{k,\mu}$ be the blow-up of $T_{k,\mu}$. By the same argument as in the proof of Theorem \ref{ThmExample1.1}, one can immediately get the diameter of $H_{k,\mu}$ as follows.

\begin{lemma}\label{lemDiamBlowGraph}
$\textup{diam}(H_{k,\mu})=\mu(3k-5)-2$.
\end{lemma}

We are now ready to complete the odd-case construction. 
\begin{theorem}\label{ThmExample1.2}
 Let $\mu\ge 2$, $r\ge 4$, $k=2r-1$, and $\delta\ge 5k^4$ be integers. Then there exists a connected $k$-colorable (hence $K_{k+1}$-free) graph $G_{k,\mu}$ with minimum degree $\delta$, order $n=\mu w$, and diameter $\mu(3k-5)-2$. Moreover, for each fixed pair $(k,\delta)$, $\textup{diam}(G_{k,\mu})-\frac{3k-2}{k}\cdot\frac{n}{\delta}\to+\infty$ as $\mu\to+\infty$. 
\end{theorem}
\begin{proof}
Set $\theta=\mu(3k-5)-1$, and let $L_1',\ldots,L_{\theta}'$ be the blown-up layers of $H_{k,\mu}$. Let $c=\left\lceil\frac{\theta}{2}\right\rceil$. Choose any vertex $x\in L_{c}'$. By Lemma \ref{lemTkWeights},  $|N_{H_{k,\mu}}(x)|\ge \delta$. Then choose a set $X\subseteq N_{H_{k,\mu}}(x)$ of size $\delta$. Let $G_{k,\mu}$ be the graph obtained from $H_{k,\mu}$  by adding a new vertex $y$ such that $N_{G_{k,\mu}}(y)=X$. 

As in the proof of Theorem \ref{ThmExample1.1}, $H_{k,\mu}$ is connected and has the same chromatic number as $T_{k,\mu}$. It is therefore $k$-colorable by Lemma \ref{lemTkChromatic}. Since every vertex of $X$ is adjacent to $x$, assigning $y$ the same color as $x$ preserves a proper $k$-coloring.  Thus $G_{k,\mu}$ is $k$-colorable and hence $K_{k+1}$-free. Lemma \ref{lemTkWeights} shows that the minimum degree of $H_{k,\mu}$ is at least $\delta$, while $d_{G_{k,\mu}}(y)=\delta$. Hence $\delta(G_{k,\mu})=\delta$. Lemma \ref{lemSum1.2Weights} gives $|V(H_{k,\mu})|=\mu w-1$. So $n=|V(G_{k,\mu})|=\mu w$.

Let $\beta=\textup{diam}(H_{k,\mu})=\theta -1$. We first prove that adding $y$ does not change any distance between old vertices. Let $u,v\in V(H_{k,\mu})$. Since $H_{k,\mu}$ is a subgraph of $G_{k,\mu}$, $d_{G_{k,\mu}}(u,v)\le d_{H_{k,\mu}}(u,v)$. Conversely, take a shortest $u$--$v$ path in $G_{k,\mu}$. If it avoids $y$, it is already a path in $H_{k,\mu}$. If it uses $y$, it contains a segment $u_1$--$y$--$u_2$ with $u_1,u_2\in X$. Replacing this segment by $u_1$--$x$--$u_2$ gives a walk of the same length in $H_{k,\mu}$; deleting any closed subwalk produces a path no longer than the original one. Therefore $d_{H_{k,\mu}}(u,v)\le d_{G_{k,\mu}}(u,v)$. Thus equality holds for all old vertices $u,v$. In particular, two vertices in the end layers remain at distance $\beta$, and hence $\textup{diam}(G_{k,\mu})\ge \beta$. 

Because $x\not\in X$ and $X\neq \emptyset$, we have $d_{G_{k,\mu}}(y,x)=2$. For an old vertex $u$ in a layer with index $i\ne c$, \[d_{H_{k,\mu}}(x,u)=|i-c|\le\textup{max}\{c-1,\theta-c\}=\left\lceil\frac{\beta}{2}\right\rceil.\]
If $u$ lies in $L_c'$ but in a blown-up class different from that of $x$, then $d_{H_{k,\mu}}(x,u)=1$. If $u$ lies in the same blown-up class as $x$, then $d_{H_{k,\mu}}(x,u)\le 2$. Since $k\ge 7$ and $\mu\ge 2$ imply $\beta\ge 30$, all three cases give $d_{H_{k,\mu}}(x,u)\le \left\lceil\frac{\beta}{2}\right\rceil$. Therefore,
 \[d_{G_{k,\mu}}(y,u)\le d_{G_{k,\mu}}(y,x)+d_{H_{k,\mu}}(x,u)\le 2+\left\lceil\frac{\beta}{2}\right\rceil\le \beta.\]
 Together with Lemma \ref{lemDiamBlowGraph},  this proves $\textup{diam}(G_{k,\mu})=\beta=\mu(3k-5)-2$.

It remains to compare $\beta$ with the conjectured expression. Let $f_k=k-1-\frac{2}{3(k-1)}$, $h_k=k^2-k+1$, and $\gamma_k=(3k-5)-\frac{3k-2}{k}f_k$. Since  $\lceil s\rceil<s+1$ for every real number $s$, the definition of $w$ gives $w\le \delta f_k+h_k$. A direct simplification yields $\gamma_k=\frac{2}{3k(k-1)}$.  Since $k\ge 7$, we have $\frac{3k-2}{k}<3$ and $h_k<k^2$. Since $\delta\ge 5k^4$, it follows that \[\frac{3k-2}{k}\cdot\frac{h_k}{\delta}<\frac{3k^2}{5k^4}<\frac{2}{3k^2}<\gamma_k.\]
Consequently, 
\begin{align*}
\textup{diam}(G_{k,\mu})-\frac{3k-2}{k}\cdot\frac{n}{\delta}&=\mu\left(3k-5-\frac{(3k-2)w}{k\delta}\right)-2\\
&\ge \mu\left(3k-5-\frac{(3k-2)(\delta f_k+h_k)}{k\delta}\right)-2\\
&\ge \mu\left(\gamma_k-\frac{3k-2}{k}\cdot\frac{h_k}{\delta}\right)-2.
\end{align*}
The coefficient of $\mu$ is positive, so the displayed difference tends to $+\infty$ as $\mu\to+\infty$. 
\end{proof}

\begin{remark}
We still do not attempt to optimize the lower bound of $\delta$. In fact, a similar construction yields counterexamples to Conjecture \ref{Conj2021Czaba}  for every $k=2r-1\ge 13$ and every $\delta\ge 25k^2$. For brevity, we omit the details.
\end{remark}

\section*{Declarations}
The authors declare that they have no known competing financial interests or personal relationships that could have appeared to influence the work reported in this paper.
\section*{\bf\Large Availability of Data and Materials}  
Not applicable.
\section*{Acknowledgments}
This research is supported by National Key R\&D Program of China under grant number 2024YFA1013900, NSFC under grant numbers 12471327 and 12401454, Natural Science Foundation of Fujian Province under grant number 2024J01875, Science-Technology Foundation of Putian University under grant number 2023059.


\end{document}